\documentclass[12pt]{amsart}
\usepackage{amssymb,amscd,amsmath}
\usepackage{amsaddr}
\usepackage{fullpage}
\usepackage{tikz-cd}
\usepackage{todonotes}
\usepackage{verbatim} 
\usetikzlibrary{intersections}

\usepackage{amsthm}
\theoremstyle{definition}
\newtheorem{theorem}{Theorem}[section]
\newtheorem{lemma}[theorem]{Lemma}
\newtheorem{proposition}[theorem]{Proposition}
\newtheorem{corollary}[theorem]{Corollary}

\theoremstyle{definition}

\newtheorem{question}[theorem]{Question}

\theoremstyle{remark}
\newtheorem{remark}[theorem]{Remark}

	\newcommand{\R}{\mathbb{R}}
	\newcommand{\N}{\mathbb{N}}
	\newcommand{\C}{\mathbb{C}}
	\newcommand{\Z}{\mathbb{Z}}
	\newcommand{\Q}{\mathbb{Q}}
	\newcommand{\h}{\mathcal{H}}

	\newcommand{\GL}{\mathrm{GL}}
	\newcommand{\SL}{\mathrm{SL}}
	\newcommand{\Stab}{\mathrm{Stab}}
	
	\DeclareMathOperator{\Aut}{Aut}
	\DeclareMathOperator{\Gal}{Gal}
	
	\DeclareMathOperator{\lcm}{lcm}

    \renewcommand{\mod}{\text{ mod}\,}
	\newcommand{\abcd}[4]{\begin{pmatrix}#1&#2\\#3& #4\end{pmatrix}}
	\newcommand{\sabcd}[4]{\left(\begin{smallmatrix}#1&#2\\#3& #4\end{smallmatrix}\right)}

	\title{Fourier expansions at cusps}
\author{Fran\c{c}ois Brunault}
\email{francois.brunault@ens-lyon.fr}
\address{\'{E}NS Lyon, UMPA, 46 all\'{e}e d'Italie, 69007 Lyon, France}
\author{Michael Neururer}
\email{neururer@mathematik.tu-darmstadt.de}
\address{TU Darmstadt, Schlo\ss gartenstr. 7, 64289 Darmstadt, Germany}
\thanks{The second author was partially funded by the DFG-Forschergruppe 1920 and the LOEWE research unit ``Uniformized Structures in Arithmetic and Geometry''}
\usepackage{hyperref}
\hypersetup{pdfauthor={Fran\c{c}ois Brunault and Michael Neururer},%
	            pdftitle={Fourier expansions at cusps},%
	            pdfsubject={Number Theory},%
	            pdfcreator={XeLaTeX},%
	            colorlinks=true,%
	            linkcolor=[rgb]{0,.6,1},
	            urlcolor = blue,
	            citecolor=blue}

\begin{document}

\begin{abstract}
In this article we study the fields generated by the Fourier coefficients of modular forms at arbitrary cusps. We prove that these fields are contained in certain cyclotomic extensions of the field generated by the Fourier coefficients at infinity. We also show that this bound is tight in the case of newforms with trivial Nebentypus. The main tool is a result of Shimura on the interplay between the actions of $\GL_2^+(\Q)$ and $\Aut(\C)$ on modular forms.
\end{abstract}

\maketitle

\section{Introduction}
In this article we study the fields generated by the Fourier coefficients of modular forms at the cusps of $X_1(N)$. To do this we study the connections between two actions on spaces of modular forms: the action of $\GL^+_2(\Q)$ via the slash-operator and the action of $\Aut(\C)$ on the Fourier coefficients of a modular form. A detailed study of these actions was conducted by Shimura in \cite{shimura1975}, where he proved a formula for the action of $\Aut(\C)$ on $f|g$. We provide a new proof of this result using a theorem of Khuri-Makdisi \cite{khuri-makdisi2012} on products of Eisenstein series. A second new proof from the perspective of Katz's theory of algebraic modular forms \cite{katz-interpolation}, is available on the arXiv \cite{brunaultneururer:NoFAC}.

We use this theorem to bound the fields generated by the Fourier coefficients of modular forms at the cusps. Let us assume for simplicity that $f$ is a modular form in $M_k(\Gamma_0(N))$, and let $g=\sabcd ABCD \in \SL_2(\Z)$. We show in Theorem \ref{thm:number fields general} that the coefficients of $f|g$ lie in the cyclotomic extension $K_f(\zeta_{N'})$, where $K_f$ is the field generated by the coefficients of $f$, and $N'=N/\gcd(CD,N)$. In the case $f$ has non-trivial Nebentypus, we show in Theorem \ref{thm:characters} that the coefficients of $f|g$ belong to a $1$-dimensional $K_f(\zeta_{N'})$-vector space, which is itself contained in an explicit cyclotomic extension $K_f(\zeta_M)$.

We apply these results in Section \ref{section:al ops} to find number fields that contain the Atkin--Lehner pseudo-eigenvalues of a newform, recovering a result of Cohen in \cite{cohen2018}.

In Section \ref{section:optimizing} we discuss how to choose $g$ among the matrices in $\SL_2(\Z)$ that map $\infty$ to a given cusp $\alpha \in \mathbb{P}^1(\Q)$ so that $N'$ (or $M$) becomes minimal. Assuming that $f\in M_k(\Gamma_0(N))$ is an eigenfunction of the Atkin--Lehner operators, we describe how to further reduce $N'$ by replacing $\alpha$ with its image under a suitable Atkin--Lehner operator. The Fourier expansion of $f|g$ can then easily be obtained from another Fourier expansion $f|g'$ which has coefficients in the field $K_f(\zeta_{\operatorname{gcd}(\delta,N/\delta)})$, where $\delta=\gcd(C,N)$ is the denominator of the cusp $\alpha=A/C$. Note that $\Q(\zeta_{\operatorname{gcd}(\delta,N/\delta)})$ is the field of definition of the cusp $\alpha$ in the canonical model of $X_0(N)$ over $\Q$.

In the last section, we prove in Theorem \ref{thm:exact nf Gamma0} that if $f$ is a newform for $\Gamma_0(N)$ then the number field provided by Theorem \ref{thm:number fields general} is the best possible, in the sense that it is the number field generated by the coefficients of $f|g$. 

Recently three algorithms for the computation of the Fourier expansion of $f|g$ have appeared: two algorithms in SageMath, one by Dan Collins \cite{collins2018} and another by Martin Dickson and the second author \cite{dickson-neururer}. The third algorithm was implemented in PARI/GP by Karim Belabas and Henri Cohen \cite{cohen2018}. While the first algorithm only uses numerical approximations of the Fourier coefficients in order to compute Petersson inner products, the latter two calculate the Fourier coefficients as algebraic numbers. The knowledge of the field (or vector space) generated by the Fourier coefficients of $f|g$ could provide a significant speed-up for these calculations.

\textbf{Acknowledgements:} We thank Henri Cohen for encouraging us to write this article and the referee for comments and corrections that led to an improvement of this article.

We are grateful to Abhishek Saha for pointing out an alternative approach to the Theorems in Section \ref{section:bounding field}, sketched in Remark \ref{remark:local Whittaker newforms}.

\subsection*{Notations} We denote by $\mathcal{H} = \{\tau \in \C : \mathrm{Im}(\tau)>0\}$ the upper-half plane, and by $\GL_2^+(\R)$ the group of $2 \times 2$ matrices with positive determinant. For any integer $k \in \Z$, we define the weight $k$ action of $\GL_2^+(\R)$ on functions $f : \h \to \C$ by
\[
f|_k g(\tau) = \frac{\det(g)^{k/2}}{(c\tau + d)^k} f\Bigl(\frac{a\tau+b}{c\tau+d}\Bigr) \qquad \left(g=\abcd abcd \in \GL_2^+(\R)\right).
\]
We will usually omit $k$ from the notation and just write $f|g$ for $f|_k g$.

The group $\Aut(\C)$ of automorphisms of $\C$ acts on modular forms as follows: for any modular form $f$ with Fourier expansion $f(\tau)=\sum_n a_n e^{2\pi i n\tau/w}$, we let
\[
f^\sigma(\tau) = \sum_n \sigma(a_n) e^{2\pi in\tau/w} \qquad (\sigma \in \Aut(\C)).
\]

For any integer $N \geq 1$, we denote $\zeta_N = e^{2\pi i/N} \in \C$.

\section{Eisenstein series}\label{section:eisenstein series}

\subsection{Definitions} We refer the reader to \cite[\S 3]{kato} for more details on Eisenstein series.

For integers $k \geq 1$, $N \geq 1$ and $a,b\in\Z/N\Z$, define the series
\[
E^{(k)}_{a,b}(\tau) = \frac{(k-1)!}{(-2\pi i)^k} \sum_{\substack{\omega \in \Z\tau+\Z \\ \omega \neq - (\tilde{a}\tau+\tilde{b})/N}} \frac{1}{(\omega+\frac{\tilde{a}\tau+\tilde{b}}{N})^k |\omega+\frac{\tilde{a}\tau+\tilde{b}}{N}|^{2s}} \Biggl. \Biggr|_{s=0}
\]
where $\tilde{a},\tilde{b}$ denote any representatives of $a,b$ in $\Z$, and $\cdot |_{s=0}$ denotes analytic continuation to $s=0$ (this is needed only when $k \in \{1,2\}$). It follows from the definition that the weight $k$ action of $\SL_2(\Z)$ on these series is given by $E^{(k)}_{a,b} | g = E^{(k)}_{(a,b)g}$ for every matrix $g \in \SL_2(\Z)$. In particular, the function $E^{(k)}_{a,b}$ is modular of weight $k$ with respect to the principal congruence subgroup $\Gamma(N)$. If $k \neq 2$, then $E^{(k)}_{a,b}$ is a holomorphic Eisenstein series of weight $k$ with respect to $\Gamma(N)$. If $k=2$, then $\tilde{E}^{(2)}_{a,b} := E^{(2)}_{a,b}-E^{(2)}_{0,0}$ is a holomorphic Eisenstein series of weight $2$ with respect to $\Gamma(N)$. 

\subsection{Fourier expansions}\label{subsection:fourier exp} We refer the reader to \cite[\S 3]{katz-interpolation} and \cite[Chap. VII]{schoeneberg} for proofs of the following facts.

If $k \neq 2$, then the Fourier expansion of $E^{(k)}_{a,b}$ is given by
\[
E^{(k)}_{a,b}(\tau) = a_0(E^{(k)}_{a,b}) + \sum_{\substack{m,n\geq 1 \\ m \equiv a (N)}} \zeta_N^{bn} n^{k-1} q^{mn/N} + (-1)^k \sum_{\substack{m,n\geq 1 \\ m \equiv -a (N)}} \zeta_N^{-bn} n^{k-1} q^{mn/N} \qquad (q^{1/N}=e^{2\pi i\tau/N}).
\]
If $k=2$, then the Fourier expansion of $\tilde{E}^{(2)}_{a,b}$ is given by
\[
\tilde{E}^{(2)}_{a,b}(\tau) = a_0(\tilde{E}^{(2)}_{a,b}) + \sum_{\substack{m,n\geq 1 \\ m \equiv a (N)}} \zeta_N^{bn} n q^{mn/N} + \sum_{\substack{m,n\geq 1 \\ m \equiv -a (N)}} \zeta_N^{-bn} n q^{mn/N} - 2 \sum_{m,n\geq 1} n q^{mn}.
\]
The constant terms $a_0(E^{(k)}_{a,b})$ and $a_0(\tilde{E}^{(2)}_{a,b})$ are elements of $\Q(\zeta_N)$ and are given in \cite[3.10]{kato} or \cite[\S 3]{brunault:regRZ}. We will not need the precise expressions since modularity determines them uniquely.

\begin{proposition}\label{proposition:Galois action on E}
Let $g=\sabcd ABCD\in\SL_2(\Z)$ and $\sigma \in \Aut(\C)$ such that $\sigma(\zeta_N)=\zeta_N^\lambda$ with $\lambda \in (\Z/N\Z)^\times$. If $k \neq 2$, then
\begin{align*}
(E^{(k)}_{a,b} | g)^{\sigma} = (E^{(k)}_{a,b})^\sigma | g_\lambda,
\end{align*}
where $g_\lambda$ is any lift in $\SL_2(\Z)$ of the matrix $\sabcd{A}{\lambda B}{\lambda^{-1} C}{D} \in \SL_2(\Z/N\Z)$. If $k=2$, then the same statement holds with $E^{(2)}_{a,b}$ replaced by $\tilde{E}^{(2)}_{a,b}$.
\end{proposition}

\begin{proof}
By explicit examination of the Fourier expansion we see $(E^{(k)}_{a,b})^\sigma = E^{(k)}_{a,\lambda b}$, so that
\begin{equation*}
(E^{(k)}_{a,b} | g)^{\sigma} = E^{(k)}_{aA+bC,\lambda(aB+bD)} = E^{(k)}_{a,\lambda b} \Bigl| \begin{pmatrix} A & \lambda B \\ \lambda^{-1}C & D \end{pmatrix} = (E^{(k)}_{a,b})^\sigma | g_\lambda.
\end{equation*}
The argument for $\tilde{E}^{(2)}_{a,b}$ is similar.
\end{proof}

\subsection{Khuri-Makdisi's Theorem}\label{subsec khuri-makdisi thm}

We recall here Khuri-Makdisi's result \cite{khuri-makdisi2012} giving generators of the graded algebra of modular forms on $\Gamma(N)$. Let $\mathcal{R}_N$ be the subalgebra of $M_*(\Gamma(N))=\bigoplus_{k\geq 0} M_k(\Gamma(N))$ generated by the Eisenstein series $E^{(1)}_{a,b}$ with $a,b\in\Z/N\Z$.

\begin{theorem}\cite{khuri-makdisi2012}\label{thm:khuri-makdisi}
If $N \geq 3$, then $\mathcal{R}_N$ contains all modular forms on $\Gamma(N)$ of weight $2$ and above. In other words, $\mathcal{R}_N$ misses only the cusp forms of weight $1$ on $\Gamma(N)$.
\end{theorem}

This follows from combining \cite[Theorem 3.5, Remark 3.14, Theorem 5.1]{khuri-makdisi2012}. The link between our notations and Khuri-Makdisi's notations is
\begin{equation*}
E^{(1)}_{a,b}(\tau) = -\frac{1}{2\pi i} G_1\Bigl(\tau,\frac{a\tau+b}{N}\Bigr) = \frac{1}{2\pi i} \lambda_{(a\tau+b)/N},
\end{equation*}
see \cite[Definition 2.1 and Corollary 3.13]{khuri-makdisi2012}.

\begin{remark}
If $N=2$, then the algebra $M_*(\Gamma(2))$ is generated by the weight $2$ Eisenstein series $\tilde{E}^{(2)}_{1,0}$, $\tilde{E}^{(2)}_{0,1}$ and $\tilde{E}^{(2)}_{1,1}$, the only relation being $\tilde{E}^{(2)}_{1,0} + \tilde{E}^{(2)}_{0,1} + \tilde{E}^{(2)}_{1,1} = 0$, see \cite[Remark 3.6]{khuri-makdisi2012}.
Of course, if $N=1$ then $M_*(\SL_2(\Z))$ is freely generated by the usual Eisenstein series of weight $4$ and $6$.
\end{remark}

\section{The actions of $\SL_2(\Z)$ and $\mathrm{Aut}(\C)$ on modular forms}\label{section:GL_2(Z/NZ)-action}

In this section we investigate the connection between the natural actions of $\SL_2(\Z)$ and $\Aut(\C)$ on modular forms.

The following theorem is a particular case of a result of Shimura, see \cite[Theorem 8]{shimura1975} and \cite[Lemma 10.5]{shimura2000}. Shimura proves this theorem in a much greater generality, for meromorphic vector-valued Siegel modular forms. We will content ourselves with the special case of holomorphic modular forms:

\begin{theorem}\cite{shimura1975,shimura2000}\label{fg sigma}
Let $f\in M_k(\Gamma(N))$ be a modular form of weight $k \geq 1$ on $\Gamma(N)$. Let $g = \sabcd ABCD \in\SL_2(\Z)$ and $\sigma\in\Aut(\C)$ such that $\sigma(\zeta_N)=\zeta_N^\lambda$ with $\lambda \in (\Z/N\Z)^\times$. Then
\[
(f|g)^{\sigma} = f^{\sigma} | g_\lambda,
\]
where $g_\lambda$ is any lift in $\SL_2(\Z)$ of the matrix $\sabcd{A}{\lambda B}{\lambda^{-1} C}{D} \in \SL_2(\Z/N\Z)$.
\end{theorem}

\begin{proof}
Let us assume $N \geq 3$, and let $f\in \mathcal{R}_N$. The maps $h \mapsto (h | g)^\sigma$ and $h \mapsto h^\sigma | g_\lambda$ are both $\sigma$-linear ring homomorphisms. We are thus reduced to the case $f=E^{(1)}_{a,b}$, which follows from Proposition \ref{proposition:Galois action on E}. If $f \in S_1(\Gamma(N))$, then $f^2$ and $f^3$ are in $\mathcal{R}_N$, so Theorem \ref{fg sigma} holds for them. Using $f = f^3/f^2$, we get $(f|g)^\sigma = f^\sigma | g_\lambda$. In the case $N=2$, we proceed similarly by applying Proposition \ref{proposition:Galois action on E} to $\tilde{E}^{(2)}_{a,b}$. Finally, the case $N=1$ is trivial.
\end{proof}

For another proof of this theorem using algebraic modular forms, see \cite{brunaultneururer:NoFAC}. If we restrict to modular forms on $\Gamma_1(N)$, then Theorem \ref{fg sigma} also follows from the result of Borisov and Gunnells \cite[Thm 5.15]{BG3} that all modular forms of sufficiently large weight are toric.

By Theorem \ref{fg sigma}, the space $M_k(\Gamma(N);\Q(\zeta_N))$ of modular forms with coefficients in $\Q(\zeta_N)$ is stable under the weight $k$ action of $\SL_2(\Z)$. It is thus endowed with a right action of $\SL_2(\Z/N\Z)$. The Galois group $\Gal(\Q(\zeta_N)/\Q)$ also acts on $M_k(\Gamma(N);\Q(\zeta_N))$ via the Fourier expansion. For any $f \in M_k(\Gamma(N);\Q(\zeta_N))$, let us then define
\begin{equation*}
f \Bigl| \abcd100\lambda = f^{\sigma_\lambda} \qquad (\lambda \in (\Z/N\Z)^\times),
\end{equation*}
where $\sigma_\lambda \in \Gal(\Q(\zeta_N)/\Q)$ is the automorphism defined by $\sigma_\lambda(\zeta_N)=\zeta_N^\lambda$. Then the above actions of $\SL_2(\Z/N\Z)$ and $(\Z/N\Z)^\times$ combine to give a right action of $\GL_2(\Z/N\Z)$ on $M_k(\Gamma(N);\Q(\zeta_N))$.
Indeed 
\begin{equation}\label{eq g glambda}
g \abcd 100\lambda = \abcd A{\lambda B}C{\lambda D} = \abcd 100\lambda \abcd A{\lambda B}{\lambda^{-1} C}D = \abcd 100\lambda g_\lambda
\end{equation}
and Theorem \ref{fg sigma} says precisely that both sides of this equality act in the same way on $M_k(\Gamma(N);\Q(\zeta_N))$. Note also that with this definition, the identities $E^{(k)}_{a,b} | g = E^{(k)}_{(a,b)g}$ for $k\neq 2$ and $\tilde{E}^{(2)}_{a,b} |g = \tilde{E}^{(2)}_{(a,b)g}$ are true for any $g \in \GL_2(\Z/N\Z)$.

\begin{remark}
Let $\tilde{Y}(N)$ be the model of $\Gamma(N)\backslash\mathcal{H}$ over $\Q(\zeta_N)$ constructed by Shimura in \cite[Chapter 6]{shimura1971}. The automorphism group of the $\Q$-scheme $\tilde{Y}(N)$ contains $\GL_2(\Z/N\Z)$. This gives a right action of $\GL_2(\Z/N\Z)$ on the function field $F_N$ of $\tilde{Y}(N)$, which is a subfield of $\Q(\zeta_N)((q^{1/N}))$. The action of $\SL_2(\Z/N\Z)$ is the slash action in weight $0$, and the action of the diagonal matrix $\sabcd 100\lambda$ with $\lambda \in (\Z/N\Z)^\times$ coincides with the natural action of $\sigma_\lambda$ on the Fourier coefficients. In other words, Theorem \ref{fg sigma} also holds for modular functions. This can also be checked using the functions denoted by $f_a$ in \cite[6.1.A]{shimura1971}, which generate $F_N$. Once the statement is proved for modular functions, Theorem \ref{fg sigma} for a given weight $k \geq 1$ follows by inspection on just one Eisenstein series of weight $k$.
\end{remark}

\begin{corollary}\label{Kf Qchi}
Let $f \in M_k(\Gamma_0(N),\chi)$ be a nonzero modular form of weight $k \geq 1$, level $N \geq 1$ and Nebentypus character $\chi$. Then the field $K_f$ generated by the Fourier coefficients of $f$ contains the field $\Q(\chi)$ generated by the values of $\chi$.
\end{corollary}

\begin{proof}
We have to show that every $\sigma \in \Aut(\C/K_f)$ fixes $\Q(\chi)$. Let $g \in \Gamma_0(N)$. Then $f|g=\chi(g) f$. Applying $\sigma$, we get $(f | g)^\sigma = \chi(g)^\sigma f$. But Theorem \ref{fg sigma} implies
\begin{equation*}
(f | g)^\sigma = f^\sigma | g_\lambda = f | g_\lambda = \chi(g_\lambda) f = \chi(g) f,
\end{equation*}
so that $\chi=\chi^\sigma$.
\end{proof}

Corollary \ref{Kf Qchi} can also be proved using Katz's theory of algebraic modular forms, noting that the diamond operators $\langle \delta \rangle$, $\delta \in (\Z/N\Z)^\times$ are defined over $\Q$, hence leave stable the space $M_k(\Gamma_1(N);K)$ of modular forms with coefficients in a fixed subfield $K$ of $\C$.

\section{Bounding the coefficient field of $f | g$}\label{section:bounding field}

We now apply the results in Section \ref{section:GL_2(Z/NZ)-action} to get information on the field generated by the Fourier coefficients of a modular form at a given cusp.

\begin{theorem}\label{thm:number fields general}
Let $f \in M_k(\Gamma_1(N))$ be a modular form of integral weight $k \geq 1$ on $\Gamma_1(N)$. Let $K_f$ be the subfield of $\C$ generated by the Fourier coefficients $a_n(f)$, $n \geq 1$. Let $g = \sabcd ABCD \in \SL_2(\Z)$.
\begin{enumerate}
\item The modular form $f |_k g$ has coefficients in $K_f(\zeta_M)$ with $M=N/\gcd(C,N)$.
\item If $f \in M_k(\Gamma_0(N))$ then $f |_k g$ has coefficients in $K_f(\zeta_{N'})$ with $N'=N/\gcd(CD,N)$.
\end{enumerate}
\end{theorem}

\begin{proof}
Let $\sigma \in \Aut(\C)$. By Theorem \ref{fg sigma}, a sufficient condition for $f|g$ being fixed by $\sigma$ is given by $f^{\sigma} = f = f | g_\lambda g^{-1}$, where $\sigma(\zeta_N)=\zeta_N^\lambda$. We have
\begin{equation}\label{glambda ginv}
g_\lambda g^{-1} \equiv \abcd{AD-\lambda BC}{AB(\lambda-1)}{CD(\lambda^{-1}-1)}{AD-\lambda^{-1}BC} \mod N.
\end{equation}
We see that $g_\lambda g^{-1} \in \Gamma_0(N)$ if and only if $\lambda \equiv 1\mod N'$. If $f \in M_k(\Gamma_0(N))$, then $f|g$ is fixed by every $\sigma \in \Aut(\C/K_f(\zeta_{N'}))$, hence has coefficients in $K_f(\zeta_{N'})$, which proves (2).

Furthermore $AD-\lambda BC = 1+BC(1-\lambda)$ so that $g_\lambda g^{-1}\in\Gamma_1(N)$ if and only if $\lambda \equiv 1\mod N'$ and $\lambda \equiv 1\mod N/\gcd(BC,N)$. Since $B$ and $D$ are coprime, the conjunction of these conditions is equivalent to $\lambda \equiv 1\mod N/\gcd(C,N)$. This proves (1).
\end{proof}

We now turn to modular forms with Nebentypus. We will actually bound not only the field of coefficients of $f|g$, but also the \emph{vector space} generated by the coefficients of $f|g$.

In order to state our results, we need some more notation. Let $f \in M_k(\Gamma_0(N),\chi)$, where $\chi$ is a Dirichlet character of conductor $m$ dividing $N$, and let $g = \sabcd ABCD \in \SL_2(\Z)$. Put $N'=N/\gcd(CD,N)$, $m'=m/\gcd(BC,m)$ and $M_\chi=\lcm(N',m')$. Let $K=K_f(\zeta_{N'})$ and $L=K_f(\zeta_{M_\chi})$. Since $L=K(\zeta_{m'})$, the extension $L/K$ is abelian and its Galois group $G=\Gal(L/K)$ identifies with a subgroup $G'$ of $(\Z/m'\Z)^\times$ by means of the cyclotomic character $\lambda : G \to (\Z/m'\Z)^\times$. Since $\Gal(L/K) \cong \Gal(\Q(\zeta_{m'})/K')$ with $K'=K \cap \Q(\zeta_{m'})$, the subgroup $G' \subset (\Z/m'\Z)^\times$ corresponds to the subfield $K' \subset \Q(\zeta_{m'})$.

\begin{lemma}\label{lemma chig}
The map $\chi_g : G \to \C^\times$ defined by
\begin{equation*}
\chi_g(\sigma) = \chi(AD-\lambda(\sigma)^{-1} BC) \qquad (\sigma \in G)
\end{equation*}
is a group homomorphism.
\end{lemma}

\begin{proof}
Let $\sigma \in G$. Note that $m'BC$ is divisible by $m$, so that $AD-\lambda(\sigma)^{-1}BC$ is well-defined in $\Z/m\Z$. Let $\lambda_N(\sigma) \in (\Z/N\Z)^\times$ denote the cyclotomic character modulo $N$. Since $\lambda_N(\sigma) \equiv 1 \mod N'$, the identity (\ref{glambda ginv}) shows that $g_{\lambda_N(\sigma)} g^{-1}$ is upper-triangular modulo $N$. It follows that $AD-\lambda_N(\sigma)^{-1} BC \in (\Z/N\Z)^\times$ and thus $AD-\lambda(\sigma)^{-1} BC \in (\Z/m\Z)^\times$. Therefore the map $\chi_g$ is well-defined.

Let us show that $\chi_g$ is a group homomorphism. We may write $\chi_g$ as the composition
\begin{equation*}
G \xrightarrow{\lambda} G' \xrightarrow{\psi} (\Z/m\Z)^\times \xrightarrow{\chi} \C^\times
\end{equation*}
where $\psi$ is defined by $\psi(\mu)=AD-\mu^{-1}BC$. Using the relation \eqref{eq g glambda}, we get the following identity in $\GL_2(\Z/N\Z)$
\begin{equation*}
g_{\mu \mu'} g^{-1} = \abcd100{\mu^{-1}} g_{\mu'} g^{-1} \abcd100\mu g_\mu g^{-1} \qquad (\mu,\mu' \in (\Z/N\Z)^\times).
\end{equation*}
Specialising to the case $\mu,\mu' \equiv 1 \mod N'$ and comparing the bottom-right entries, we deduce that $\psi$ is a group homomorphism.
\end{proof}

Note that the character $\chi_g$ takes values in $\Q(\chi)^\times$, which is contained in $K_f^\times$ by Corollary \ref{Kf Qchi}. By the normal basis theorem, $L$ is a free $K[G]$-module of rank 1. Since $K_f$ is contained in $K$, the character $\chi_g$ cuts out a $K$-line $L^{\chi_g}$ in $L$, namely
\begin{equation}
L^{\chi_g} = \{x \in L : \forall \sigma \in G, \sigma(x) = \chi_g(\sigma) x \}.
\end{equation}
We are now ready to state our result.

\begin{theorem}\label{fg Lchig}
The modular form $f | g$ has coefficients in $L ^{\chi_g}$.
\end{theorem}

\begin{proof}
Let $\sigma \in \Aut(\C/K)$ with $\sigma(\zeta_N)=\zeta_N^\lambda$. Since $\lambda \equiv 1 \mod N'$, we have $g_\lambda g^{-1}\in \Gamma_0(N)$. Then
\begin{equation}\label{fg sigma eq1}
(f | g)^\sigma = f^\sigma | g_\lambda = f | g_\lambda g^{-1} g = \chi(AD-\lambda^{-1}BC) f|g = \chi_g(\sigma |_L) f|g.
\end{equation}
In particular $f | g$ is fixed by $\Aut(\C/L)$, hence has coefficients in $L$. Moreover (\ref{fg sigma eq1}) shows that $f|g$ has coefficients in $L^{\chi_g}$.
\end{proof}

We summarise our result and make it slightly more precise as follows.

\begin{theorem}\label{thm:characters}
Let $f \in M_k(\Gamma_0(N),\chi)$, where $\chi$ is a Dirichlet character of conductor $m$ dividing $N$, and let $g = \sabcd ABCD \in \SL_2(\Z)$. Put $N'=N/\gcd(CD,N)$, $m'=m/\gcd(BC,m)$. Then $f |_k g$ has coefficients in $K_f(\zeta_{M_\chi})$ with
\begin{equation*}
M_\chi=\lcm(N',m') = \frac{\lcm(NB,mD,BCD)}{BCD}.
\end{equation*}
More precisely, let $G'$ be the subgroup of $(\Z/m'\Z)^\times$ corresponding to the abelian number field $K_f(\zeta_{N'}) \cap \Q(\zeta_{m'})$, and let $\zeta$ be any $m'$-th root of unity such that
\begin{equation*}
c_{\chi,g} := \sum_{\mu \in G'} \chi(AD-\mu BC) \zeta^\mu
\end{equation*}
is nonzero (such a $\zeta$ always exists). Then $f |_k g$ has coefficients in $c_{\chi,g} \cdot K_f(\zeta_{N'})$.
\end{theorem}

\begin{proof}
The fact that $f | g$ has coefficients in $L=K_f(\zeta_{M_\chi})$ was proved in Theorem \ref{fg Lchig}. Let $K=K_f(\zeta_{N'})$ and let $\pi_{\chi_g} : L \to L^{\chi_g}$ be the $K$-linear projector associated to the linear character $\chi_g : G \to K^\times$. It is given explicitly by
\[
\pi_{\chi_g}(x) = \frac{1}{|G|}\sum_{\tau\in G} \chi_g(\tau^{-1})\tau(x) \qquad (x \in L).
\]
Since $L$ is generated as a $K$-vector space by the $m'$-th roots of unity, there exists $\zeta \in \mu_{m'}$ such that $\pi_{\chi_g}(\zeta) \neq 0$. We set $c_{\chi,g}=|G|\cdot \pi_{\chi_g}(\zeta)$, so that
\begin{equation*}
c_{\chi,g} = \sum_{\tau \in G} \chi_g(\tau^{-1}) \tau(\zeta) = \sum_{\tau \in G} \chi_g(\tau^{-1}) \zeta^{\lambda(\tau)} = \sum_{\mu \in G'} \chi(AD-\mu BC) \zeta^\mu.
\end{equation*}
By Theorem \ref{fg Lchig}, the modular form $f|g$ has coefficients in $L^{\chi_g} = c_{\chi,g} \cdot K$.
\end{proof}

The second part of Theorem \ref{thm:characters} says that the coefficients of $f|g$ belong to a $K_f$-vector space which has the same dimension as in the $\Gamma_0(N)$ case. We also note that $M_\chi$ always divides the $M$ provided by Theorem \ref{thm:number fields general}, so Theorem \ref{thm:characters} provides better results than Theorem \ref{thm:number fields general} for modular forms with Nebentypus. We thank the referee for pointing out the simpler description of $M_\chi$ given in Theorem \ref{thm:characters}.

The choice $\zeta=\zeta_{m'}$ in Theorem \ref{thm:characters} does not always work. For example, take the newform $f$ of weight $k=3$, level $N=9$ and character $\chi$ of conductor $m=9$, with $\chi(4)=\zeta_3$ and $\chi(-1)=-1$. We have $K_f=\Q(\chi)=\Q(\zeta_3)$. Taking $g=\sabcd0{-1}13$, we get $N'=3$, $m'=9$ and $G'=\{1,4,7\}$, so that $c_{\chi,g}=0$ for $\zeta=\zeta_9$. On the other hand, for $\zeta=\zeta_9^2$ we get $c_{\chi,g}=3\zeta_9^2$ and $f | g$ indeed has coefficients in $\zeta_9^2 \cdot \Q(\zeta_3) = \langle \zeta_9^2, \zeta_9^5 \rangle_{\Q}$.

\begin{remark}\label{rk:composite field}
Theorem \ref{fg Lchig} also shows that the coefficients of $f|g$ lie in the fixed field $L^{\ker \chi_g}$. Let $\chi'_g : G' \to \C^\times$ be the character defined by $\chi'_g(\mu)=\chi(AD-\mu^{-1}BC)$ (using the notations of the proof of Lemma \ref{lemma chig}, we have $\chi'_g = \chi \circ \psi$ and $\chi_g = \chi'_g \circ \lambda$). Then the field $L^{\ker \chi_g}$ is equal to the composite $F \cdot K_f(\zeta_{N'})$, where $F$ is the subfield of $\Q(\zeta_{m'})$ corresponding to the kernel of $\chi'_g$.
\end{remark}

\begin{remark}\label{remark:local Whittaker newforms}
An alternative approach towards proving Theorems \ref{thm:number fields general} and \ref{thm:characters} would be to use local Whittaker newforms as in \cite{corbett-saha}. In particular Proposition 3.3 in \textit{loc.}\ \textit{cit.}\ gives an explicit formula for the Fourier coefficients of $f|g$ in terms of Whittaker newforms and the Galois action on such newforms is described in the proof of Proposition 2.17.
\end{remark}

\section{Atkin--Lehner operators}\label{section:al ops}

For a divisor $Q$ of $N$ with $\gcd(Q,N/Q)=1$ we define the Atkin--Lehner operator on $M_k(\Gamma_1(N))$ as follows. Choose $x,y,z,w\in\Z$ with $x\equiv 1\mod N/Q$ and $y\equiv 1\mod Q$ such that the matrix $W_Q = \sabcd{Qx}y{Nz}{Qw}$ has determinant $Q$. Note that $W_Q = h_Q \sabcd Q001$ with $h_Q=\sabcd{x}{y}{\frac{N}{Q}z}{Qw} \in \SL_2(\Z)$. For a modular form $f\in M_k(\Gamma_1(N))$ we have
\[
f|_k W_Q(\tau) = Q^{k/2} \left(f|_k
h_Q\right)(Q\tau).
\]
Therefore we can apply our previous results to find a module that contains the coefficients of $f|h_Q$, or equivalently the coefficients of $Q^{-k/2}f|W_Q$, and reprove a theorem of Cohen.

\begin{corollary}[Theorem 2.6 in \cite{cohen2018}]\label{cor:number field of atkin-lehner}
Let $Q$ be a maximal divisor of $N$ and let $f\in M_k(\Gamma_1(N))$ and $K_f$ be the subfield of $\C$ generated by its Fourier coefficients. Then
\begin{enumerate}
\item\label{item:atkin-lehner 1} The modular form $f|_k W_Q$ has coefficients in $Q^{k/2} \cdot K_f(\zeta_Q)$.
\item\label{item:atkin-lehner 2} If $f\in M_k(\Gamma_0(N),\chi)$ for a character $\chi$ of conductor $m$, then $f|_k W_Q$ has coefficients in $Q^{k/2}G'(\chi_Q) \cdot K_f$, where $G'(\chi_Q)$ is the Gauss sum of the primitive character associated to the $Q$-part of $\chi$.
\end{enumerate}
\end{corollary}
\begin{proof}
The first statement follows directly from Theorem \ref{thm:number fields general}.

Now let $f\in M_k(\Gamma_0(N),\chi)$. We will prove that the coefficients of $f|h_Q$ lie in $G'(\chi_Q)\cdot K_f$, which is equivalent to the second statement. First we determine the character $\chi_{h_Q}$ from the previous section. Splitting $\chi$ as a product of its $Q$-part and its $N/Q$-part we observe
\[
\chi_{h_Q}(\sigma) = \chi_Q(-\frac{N}{Q}zy\lambda(\sigma)^{-1})\chi_{N/Q}(Qwx) = \chi_Q(\lambda(\sigma)^{-1}) = \overline{\chi_Q(\sigma)}.
\]
The conclusion now follows from Theorem \ref{thm:characters}. Since $N' = 1$, $m' = m_Q$, and for $\zeta=\zeta_{m'}$ we have $c_{\chi,h_Q}=G'(\chi_Q)$.
\end{proof}

Let $f\in S_k(\Gamma_0(N),\chi)$ be a newform. Then according to \cite[\S 1]{AtkinLi78} there exists a newform $\tilde{f} \in S_k(\Gamma_0(N),\overline{\chi_Q}\chi_{N/Q})$ and an algebraic number $\lambda_Q(f)$ of absolute value $1$ such that
\begin{equation}\label{fWQ atkin-li}
f|W_Q = \lambda_Q(f) \tilde{f}.
\end{equation}
The number $\lambda_Q(f)$ is called the pseudo-eigenvalue of $f$ at $Q$. By looking at the first Fourier coefficient in (\ref{fWQ atkin-li}), we get the following result.

\begin{corollary}
If $f\in S_k(\Gamma_0(N),\chi)$ is a newform, then the pseudo-eigenvalue $\lambda_Q(f)$ is in $Q^{k/2}G'(\chi_Q) \cdot K_f$ and $\tilde{f}$ has coefficients in $K_f$.
\end{corollary}

This should be compared to the following theorem of Atkin--Li, where an explicit formula for $\lambda_Q(f)$ is derived in a special case.
\begin{theorem}\cite[Theorem 2.1]{AtkinLi78}
Let $f=\sum_n a_n e^{2\pi in\tau} \in S_k(\Gamma_0(N),\chi)$ be a newform, $q$ be a prime dividing $N$ and $Q = N_q$. If $a_q \neq 0$, then
\[
\lambda_Q(f) = Q^{k/2-1} \frac{G(\chi_Q)}{a_Q},
\]
where $G(\chi_Q) = \sum_{u \in (\Z/Q\Z)^\times} \chi_Q(u) e^{2\pi iu/Q}$ is the Gauss sum of $\chi_Q$.
\end{theorem}

\section{Optimising the coefficient field}\label{section:optimizing}

We may reduce the fields provided by Theorems \ref{thm:number fields general} and \ref{thm:characters} as follows. Let $\alpha \in \mathbb{P}^1(\Q)$ be a cusp, and let $g = \sabcd ABCD \in \SL_2(\Z)$ such that $g \infty = \alpha$. In order to compute the Fourier expansion of $f$ at $\alpha$, we may replace $g$ by $g T^u$ with $T=\sabcd1101$ and $u \in \Z$. Then $f | gT^u$ depends only on the class of $u$ modulo $w$, where $w$ is the width of the cusp $\alpha$. The following proposition gives the minimal value of the integer $M_\chi$ introduced in Theorem \ref{thm:characters} when $u$ varies in $\Z/w\Z$.

\begin{proposition}\label{lem:minimal number field - fixed cusp}
In the notation of Theorem \ref{thm:characters}, let $M'_\chi$ be the minimal value of $M_\chi$ for $gT^u$ as $u$ varies. Then
\[
M'_\chi = \frac{N_C}{\gcd(C,N)} \cdot m_{\overline{C}},
\]
where $N_C=\prod_{p|C}p^{v_p(N)}$ is the $C$-part of $N$ and $m_{\overline{C}} = m/m_C$ is the prime to $C$ part of $m$. Moreover, $M'_\chi$ is attained for any $u$ such that $N/N_C$ divides $uC+D$.
\end{proposition}
\begin{proof}
Replacing $g$ with $gT^u$ changes $D$ to $uC+D$ and $B$ to $uA+B$. Varying $u$, we need to determine the minimal value of
\[
M_{\chi,u} = \lcm\left(\frac{N}{\gcd(C(uC+D),N)},\frac{m}{\gcd(C(uA+B),m)}\right).
\]
Since $C$ and $D$ are coprime we have $\gcd(C(uC+D),N)=\gcd(C,N)\gcd(uC+D,N)$, so $N/\gcd(C(uC+D),N)$ is divisible by $N_C/\gcd(C,N)$. Therefore $N_C/\gcd(C,N)$ divides $M_{\chi,u}$.

Let $p\nmid C$. Since $C(uA+B) = A(uC+D)-1$, we have
\begin{align*}
v_p\left(\frac{N}{\gcd(C(uC+D),N)}\right) &= v_p(N) - \min(v_p(uC+D),v_p(N)),\\
v_p\left(\frac{m}{\gcd(C(uA+B),m)}\right) &= v_p(m) - \min(v_p(A(uC+D)-1),v_p(m)).
\end{align*}
If $v_p(uC+D)\neq 0$, then $v_p\left(m/\gcd(C(uA+B),m)\right)=v_p(m)$. On the other hand, if $v_p(uC+D)=0$, then $v_p\left(N/\gcd(C(uC+D),N)\right)=v_p(N)\geq v_p(m)$. In all cases, we have $v_p(M_{\chi,u}) \geq v_p(m)$, which proves that $m_{\overline{C}}$ divides $M_{\chi,u}$. It follows that $M_{\chi,u}$ is always divisible by $N_C/\gcd(C,N) \cdot m_{\overline{C}}$.

Now choose $u$ such that $N_{\overline{C}} = N/N_C$ divides $uC+D$. This is possible because $C$ and $N_{\overline{C}}$ are coprime. Then $\gcd(uC+D,N)=N_{\overline{C}}$ so that
\begin{equation*}
\frac{N}{\gcd(C(uC+D),N)} = \frac{N}{\gcd(C,N) N_{\overline{C}}} = \frac{N_C}{\gcd(C,N)}.
\end{equation*}
Moreover $\gcd(C(uA+B),m)=\gcd(A(uC+D)-1,m)$ is coprime to $N_{\overline{C}}$ and thus divides $m_C$. It follows that
\begin{equation*}
m_{\overline{C}} \mid \frac{m}{\gcd(C(uA+B),m)}
\end{equation*}
On the other hand, if $p\mid C$, then
\begin{equation*}
v_p\left(\frac{m}{\gcd(C(uA+B),m)}\right) \leq v_p\left(\frac{m}{\gcd(C,m)}\right) \leq v_p\left(\frac{N_C}{\gcd(C,N)}\right).
\end{equation*}
Hence $M_{\chi,u} = (N_C/\gcd(C,N)) \cdot m_{\overline{C}}$.
\end{proof}

In practice, we may further reduce the field of coefficients as follows. Let $f \in M_k(\Gamma_0(N))$ be an eigenvector of the Atkin--Lehner operators and $g = \sabcd ABCD \in \SL_2(\Z)$. The denominator of the cusp $\alpha=g\infty = A/C$ of $X_0(N)$ is $\delta := \gcd(C,N)$.

Now let $Q$ be a maximal divisor of $N$, and let $W_Q$ be the associated Atkin--Lehner involution of $X_0(N)$. Using the notations of Section \ref{section:al ops}, if $f$ is an eigenvector of $W_Q$ with eigenvalue $\lambda_Q(f) \in \{\pm 1\}$, then we may write
\begin{align*}
f | g &= \lambda_Q(f) f | W_Q g = \lambda_Q(f) f | h_Q\abcd Q001 \abcd ABCD\\
&= \lambda_Q(f) f| h_Q \abcd {\frac{AQ}{\gcd(C,Q)}}s{\frac{C}{\gcd(C,Q)}}{r}\abcd {\gcd(C,Q)}{rBQ-sD}{0}{\frac{Q}{\gcd(C,Q)}},
\end{align*}
where $r,s$ are chosen, so that $r\frac{AQ}{\gcd(C,Q)} - s\frac{C}{\gcd(C,Q)}=1$.

The action of the upper-triangular matrix $\sabcd {\gcd(C,Q)}{rBQ-sD}{0}{\frac{Q}{\gcd(C,Q)}}$ on Fourier expansions is easily calculated. We now try to find $Q$ such that $f|g'$ has coefficients in the minimal possible field, where $g'=h_Q\sabcd {\frac{AQ}{\gcd(C,Q)}}s{\frac{C}{\gcd(C,Q)}}{r}\in\SL_2(\Z)$.

 Let $\delta = \delta_Q \delta_{\overline{Q}}$ where $\delta_Q$ is the $Q$-part of $\delta$. Then the cusp $\alpha'=g'\infty = W_Q \alpha$ has denominator $\delta' = \frac{Q}{\delta_Q} \delta_{\overline{Q}}$ and we may choose $Q$ such that $M':=N_{\delta'}/\delta'$ is minimal. Explicitly, the choice
 \begin{equation*}
 Q=\prod_{\substack{p | N \\ 0<v_p(\delta)\leq v_p(N)/2}} p^{v_p(N)}
 \end{equation*}
gives the minimal value $M'=\gcd(\delta,N/\delta)$. By Proposition \ref{lem:minimal number field - fixed cusp}, there exists $v \in \Z$ such that the form $f | g' T^v$ has coefficients in $K_f(\zeta_{M'})$.
Thus the problem of finding the Fourier expansion of $f|g$ reduces to finding the eigenvalue $\lambda_Q(f)\in\{\pm1\}$, calculating the Fourier expansion of $f|g'T^v$ which is over a potentially much smaller field than that of $f|g$, and finally applying an upper-triangular matrix to $f|g'T^v$. Some information, such as the vanishing order of $f|g$ or the absolute value of its Fourier coefficients, can be extracted directly from $f|g'T^v$ without further calculation.

\section{Determining the exact coefficient field}\label{section:determining number field}
Our final goal is to determine the exact coefficient field of $f|g$ when $f$ is a newform. In this section, we assume that $f$ is a newform of (even) weight $k$ on the group $\Gamma_0(N)$.

We will need the following theorems of Newman \cite{newman1955,newman} on the congruence subgroup $\Gamma_0(N)$, where $N$ is a fixed integer $\geq 1$.

\begin{theorem}\cite[Theorem 3]{newman1955}\label{newman1}
Every intermediate subgroup between $\Gamma_0(N)$ and $\SL_2(\Z)$ is of the form $\Gamma_0(M)$ for some positive divisor $M$ of $N$.
\end{theorem}

In the following, we denote by $R$ the matrix $\sabcd1011$.

\begin{corollary}\label{Gamma0M RM}
Let $M$ be a positive divisor of $N$. The group $\Gamma_0(M)$ is generated by $\Gamma_0(N)$ and $R^M=\sabcd10M1$.
\end{corollary}

\begin{proof}
Let $\Gamma$ be the group generated by $\Gamma_0(N)$ and $R^M$. By Theorem \ref{newman1}, we have $\Gamma=\Gamma_0(M')$ for some $M'$ dividing $N$. Since $R^M \in \Gamma_0(M')$, the integer $M'$ divides $M$. Moreover $\Gamma$ is contained in $\Gamma_0(M)$, so that $M$ divides $M'$. It follows that $\Gamma=\Gamma_0(M)$.
\end{proof}

\begin{theorem}\cite{newman}\label{newman2}
The normaliser of $\Gamma_0(N)$ in $\SL_2(\Z)$ is equal to $\Gamma_0(N/s)$, where $s$ is the largest divisor of $24$ such that $s^2$ divides $N$. Moreover, the quotient group $\Gamma_0(N/s)/\Gamma_0(N)$ is cyclic of order $s$, generated by the class of $R^{N/s} = \sabcd10{N/s}1$.
\end{theorem}

\begin{proof}
The first assertion follows from \cite[Theorem 1]{newman}. By Corollary \ref{Gamma0M RM}, the group $\Gamma_0(N/s)$ is generated by $\Gamma_0(N)$ and $R^{N/s}$. It follows that the quotient group $\Gamma_0(N/s)/\Gamma_0(N)$ is generated by the class of $R^{N/s}$, and it is easy to see that this class has order $s$.
\end{proof}

\begin{proposition}\label{Stab f}
Let $F$ be a nonzero element of the new subspace $S_k^{\textrm{new}}(\Gamma_0(N))$. Then its stabiliser
\begin{equation*}
\Stab(F) = \{g \in \SL_2(\Z) : F | g = F\}
\end{equation*}
is equal to $\Gamma_0(N)$.
\end{proposition}

\begin{proof}
Since $\Stab(F)$ contains $\Gamma_0(N)$, Theorem \ref{newman1} implies that $\Stab(F)=\Gamma_0(M)$ for some positive divisor $M$ of $N$. But $F$ belongs to the new subspace, so we must have $M=N$.
\end{proof}

\begin{proposition}\label{fg fsigma}
Let $f$ be a newform of weight $k$ on $\Gamma_0(N)$. Let $g \in \SL_2(\Z)$ and $\sigma \in \Aut(\C)$ such that $f | g = f^\sigma$. Then we have $f^\sigma=f$ and $g \in \Gamma_0(N)$.
\end{proposition}

\begin{proof}
By Proposition \ref{Stab f}, the stabilisers of $f$ and $f^\sigma$ are both equal to $\Gamma_0(N)$. On the other hand $\Stab(f | g) = g^{-1} \Stab(f) g$, so that $g$ normalises $\Gamma_0(N)$. By Theorem \ref{newman2}, we have $g \in \Gamma_0(N/s)$, and there exists an integer $m \in \Z$ such that $\Gamma_0(N) g = \Gamma_0(N) R^{mN/s}$. Hence $f | g = f | R^{mN/s}$.

We now make use of the Atkin--Lehner involution $W_N = \sabcd 0{-1}N0$. Let $w \in \{\pm 1\}$ be the root number of $f$, defined by $f | W_N = w f$. Since $W_N$ is defined over $\Q$, we also have $f^\sigma | W_N = w f^\sigma$. Applying $W_N$ on both sides of the equality $f | g = f^\sigma$, we get
\begin{equation*}
w f^\sigma = f^\sigma | W_N = f|gW_N = f | R^{mN/s} W_N = f | W_N \left(W_N^{-1} R^{mN/s} W_N \right) = w f | R'
\end{equation*}
with $R'=W_N^{-1} R^{mN/s} W_N = \sabcd1{-m/s}01$. The Fourier expansion of $f | R'$ is given by
\begin{equation*}
f | R'(z) = f(z-m/s) = \sum_{n \geq 1} a_n(f) e^{-2\pi imn/s} e^{2\pi inz}.
\end{equation*}
Comparing the first term of the Fourier expansions, we get $e^{-2\pi im/s}=1$. This implies that $s$ divides $m$, $f^\sigma = f$ and $g \in \Gamma_0(N)$.
\end{proof}

We are now in the position to determine the exact number field of $f | g$. This refines Theorem \ref{thm:number fields general}(2) for $\Gamma_0(N)$ newforms.

\begin{theorem}\label{thm:exact nf Gamma0}
Let $f$ be a newform of weight $k$ on $\Gamma_0(N)$. Let $g=\sabcd ABCD \in \SL_2(\Z)$. Then the field generated by the Fourier coefficients of $f|_k g$ is equal to $K_f(\zeta_{N'})$ with $N'=N/\gcd(CD,N)$.
\end{theorem}

\begin{proof}
We have to show that every automorphism of $\C$ fixing $f|g$ also fixes $K_f(\zeta_{N'})$. Let $\sigma \in \Aut(\C)$ such that $(f | g)^\sigma = f |g$. Define $\lambda \in (\Z/N\Z)^\times$ by $\sigma(\zeta_N) = \zeta_N^\lambda$. By Theorem \ref{fg sigma}, we have $f^\sigma | g_\lambda = f | g$, so that $f^\sigma = f | g g_\lambda^{-1}$. By Proposition \ref{fg fsigma}, we have $f^\sigma = f$ and $g g_\lambda^{-1} \in \Gamma_0(N)$. It follows that $\sigma$ fixes $K_f$, and we have already seen in the proof of Theorem \ref{thm:number fields general} that the condition $g g_\lambda^{-1} \in \Gamma_0(N)$ is equivalent to $\lambda \equiv 1 \mod{N'}$. Therefore $\sigma$ fixes $K_f(\zeta_{N'})$.
\end{proof}

\begin{remark}
An inspection of the proofs shows that Proposition \ref{fg fsigma} and Theorem \ref{thm:exact nf Gamma0} are valid for elements $f=\sum a_n q^n$ of the new subspace of $S_k(\Gamma_0(N))$ that are eigenfunctions of $W_N$ and satisfy the following condition: if $s$ denotes the largest divisor of $6$ whose square divides $N$, then there exists $n\in\N$ which is coprime to $s$ such that $a_n$ is a non-zero rational number. One family of such forms is given by the traces $\sum_\sigma f^\sigma$ of newforms $f$, where the sum is over all embeddings of $K_f$ into $\C$.
\end{remark}

\begin{question}
It would be interesting to extend Theorem \ref{thm:exact nf Gamma0} to newforms with non-trivial Nebentypus. Is the field provided by Remark \ref{rk:composite field} best possible?
\end{question}

\begin{question}
What is the $\Q$-vector space (or $K_f$-vector space) generated by the Fourier coefficients of $f | g$?
\end{question}

\begin{question}
Can we bound the denominators of the Fourier coefficients of $f|g$? It is a classical fact that if a modular form $f \in M_k(\Gamma(N))$ has Fourier coefficients in some subring $A$ of $\C$, then for any $g \in \SL_2(\Z)$, the Fourier expansion of $f|g$ lies in $\Z[[q^{1/N}]] \otimes A[\zeta_N,1/N]$, see e.g. \cite[VII, Corollaire 3.13]{deligne-rapoport}. It would be interesting to bound the denominators effectively.
\end{question}

\bibliographystyle{plain}
\bibliography{refs}

\vspace{\baselineskip}

\end{document}